\documentclass[12pt, twoside, leqno]{article}

\usepackage{amsmath,amsthm}
\usepackage{amssymb}

\usepackage{enumerate}

\newtheorem{thm}{Theorem}[section]
\newtheorem{cor}[thm]{Corollary}
\newtheorem{lem}[thm]{Lemma}
\newtheorem{prop}[thm]{Proposition}

\theoremstyle{definition}
\newtheorem{defin}[thm]{Definition}
\newtheorem{rem}[thm]{Remark}

\numberwithin{equation}{section}

\def\P{\mathbb{P}}
\def\E{\mathbb{E}}
\def\Var{\mathrm{Var}\,}
\def\eps{\varepsilon}

\begin{document}

%%%%% To ease editing, for IMPAN journals add:

\baselineskip=17pt

%%%%%%%%%%%%%%%%

\title{Random groups are not left-orderable}

\author{Damian Orlef\\
Institute of Mathematics\\
Polish Academy of Sciences\\
00-656 Warszawa, Poland\\
E-mail: dorlef@impan.pl}

\date{}

\maketitle

%% Classification and key words; note that the 2010 classification is used:

\renewcommand{\thefootnote}{}



\renewcommand{\thefootnote}{\arabic{footnote}}
\setcounter{footnote}{0}

%%%%%%%%

\begin{abstract}
We prove that random groups in the Gromov density model
at~any density $d$ have with overwhelming probability no non-trivial left-orderable quotients.
In particular, random
groups at~densities $d<\frac{1}{2}$ are not left-orderable.
\end{abstract}

\section{Introduction}

We work in the density model for random groups introduced
by Gromov.

\begin{defin}[cf. {\cite[Section 9.B]{gromov93}},
{\cite[Definition 7]{ollivier05}}]\label{random_group_def}
Let $F_n$ be the free group on~$n \geq 2$ generators
$a_1,\ldots,a_n$. For
any integer $L$ let $R_L \subset F_n$ be~the~set of~reduced
words of length
$L$ in these generators.

Let $d\in (0,1)$. A \emph{random set of relators at density
$d$, at length $L$} is a~sequence of~$\lfloor
(2n-1)^{dL}\rfloor$
elements of $R_L$, picked independently and uniformly
at~random from all elements of $R_L$.

A \emph{random group at density $d$, at length $L$} is the
group $G$ presented by~$\langle S | R \rangle$,
where $S = \{a_1,\ldots,a_n\}$ and $R$ is a random set of
relators at~density $d$, at length $L$.
\end{defin}

The relators in $R_L$ are not assumed to be cyclically
reduced.

Of particular interest in the study of random groups are
the properties occurring \emph{with overwhelming
probability}.

\begin{defin}[cf. {\cite[Section 9.B]{gromov93}},
{\cite[Definition 7]{ollivier05}}]
Let $I\subset \mathbb{N}_+$ be~infinite. We~say that a
property of random sets of relators, or of random groups,
occurs \emph{with $I$-overwhelming
probability} (shortly,  \emph{w.\ $I$-o.p.}) \emph{at
density $d$} if its probability of~occurrence tends to $1$
as $L \rightarrow \infty$, for $L\in I$ and fixed $d$. We~omit writing ``$I$-" if~$I=\mathbb{N}_+$.
\end{defin}

Basic characteristics of the model are given by
the~following phase transition theorem, due to Gromov.

\begin{thm}[cf. {\cite[Section 9.B]{gromov93}},
{\cite[Theorem 2]{ollivier04}}]\label{phase_transition_thm}
A random group is~with overwhelming probability
\begin{itemize}
\item trivial or $\mathbb{Z}/2\mathbb{Z}$ at density
$d>\frac{1}{2}$,
\item infinite, hyperbolic and torsion-free at density
$d<\frac{1}{2}$.
\end{itemize}
\end{thm}

A number of interesting properties are known to hold for
random groups w.o.p. at~various densities (see
\cite[Section I.3]{ollivier05}).

In this paper we consider the \emph{left-orderability}.

\begin{defin}\label{left_order_def}A group $G$ is said to be
\emph{left-ordered} by $\leq$ if $\leq$ is a total order
on $G$ which is~\emph{left-invariant}: for all $g_1,g_2,h
\in G$ the condition $g_1\leq g_2$ implies
$hg_1 \leq hg_2$.\end{defin}

Our main result is the following.

\begin{thm}\label{target_thm} Let $d\in(0,1)$. A random
group in the Gromov density model at~density $d$ has w.o.p.\
no non-trivial left-orderable quotients.\end{thm}

In conjunction with Theorem \ref{phase_transition_thm},
this shows non-left-orderability of random groups below the
critical density $d=\frac{1}{2}$.

If $G$ is a~countable group (e.g. a~quotient of~a~random group), then $G$~is left-orderable if~and~only if~it admits a~faithful
action on the~real line by~orientation-preserving
homeomorphisms (see~\cite[Section 1.1.3]{deroin14}).
Theorem~\ref{target_thm} may be~thus treated as~a~result
connected to~the~Gromov conjecture
that a~random group
should not have any smooth actions on any compact manifolds
(cf. \cite[Conjecture 4.22]{furman11}).

As a~side note, Theorem~$\ref{target_thm}$
provides also an~alternative way of showing that random
groups are not free of rank~${\geq 1}$ at any density,
since free groups are left-orderable (cf. \cite[Theorem
2.3.1]{rhemtulla}). The more usual proof of~this
proceeds by~establishing that random groups have trivial
abelianizations.

The main idea of our proof is to use the~order on~the~given non-trivial quotient~$Q$ of the~random~group $G=\langle S|R\rangle$ as follows. We explicitly construct a~high-density
set $P$ of words in $F_n$, representing strictly positive
(in~the~sense of~the~order) elements of $Q$. It happens that for fixed $d$, the~density of $P$ exceeds $(1-d)$ for~$n$
sufficiently large. By~a~well-known fact it~thus contains
w.o.p.\ a word $w$ from the set $R$ of~relators, leading to
a~contradiction of~the~element corresponding to $w$
in $Q$ being both positive and~trivial.
Finally, we use the approach of \cite{przytycki11} to
increase the~number
of generators we~work with and~obtain the result for all
$n\geq 2$.

The whole proof is phrased in the language of the
b-automata and the~associated groups, as introduced in
\cite{przytycki11} and follows a very similar framework.
Just in the case of fixed $d$ and sufficiently large $n$
one can entirely avoid referring to~\cite{przytycki11}
and~provide a~bit shorter argument. It consists
of~considering sets $\mathcal{L}_{\varepsilon, i}$
from the~proof of~Lemma~\ref{positive_relator_lem},
proving they all intersect w.o.p\ a~random set of~relators
by~the~usual density argument and then proceeding
as in the proof of
Proposition~\ref{simple_non_ordered_prop}.

This paper is~structured as follows. Section 2 deals with
the basic properties of~the~left-ordered groups. In
Section~3 we introduce the notion of~a~b-automaton and its
language and use them
to give a proof of~Theorem~\ref{target_thm} for~$n$
sufficiently large. In Section 4 we use
the concept of~the~associated groups to generalise it to
all \mbox{$n\geq 2$}. In Appendix A we reprove a~well-known
generalisation of the fact that a~random set of~elements at
density $d$ intersects w.o.p.\ any fixed set of~elements of
density $d'$ such that $d+d' > 1$. The~more general
statement is that their intersection is roughly of~density
\mbox{$d+d'-1$}
if $d<\frac{1}{2}$ (cf. \cite[Section 9.A]{gromov93}).
The~assumption on $d$ is not really limiting, in~view
of~Theorem~\ref{phase_transition_thm}. It comes from the
fact that we define ``a~random set at~density~$d$"
to~be~a~tuple with possible repetitions. If we, however,
have $d<\frac{1}{2}$, then there are w.o.p.\  no such
repetitions and~the~counting is easier.

\section{Left orders}

Let $G$ be a group left-ordered by $\leq$. Symbols $<$ and
$>$ are the usual shorthands. By $e$ we~denote the neutral
element of $G$. The following remarks are easily obtained
from Definition \ref{left_order_def}.

\begin{rem}\label{ord_rmk1}Any non-empty product of
elements strictly greater than $e$ is itself strictly
greater than $e$.\end{rem}

\begin{rem}\label{ord_rmk2}For every $g\in G\setminus\{e\}$
one can choose a sign $\varepsilon \in \{-1, 1\}$ such that
$g^\varepsilon>e$.\end{rem}

Those two imply quickly the following.
\begin{cor}\label{ord_cor} $G$ is torsion-free.
\end{cor}

Moreover, by combining Remarks \ref{ord_rmk1} and
\ref{ord_rmk2}, we obtain Lemma \ref{ord_lem},
which will be used in the~paper to construct high-density
sets of words
representing non-trivial elements.

\begin{lem}\label{ord_lem}
For every choice of non-trivial $g_1,\ldots,g_n \in G$
there exists a~sequence of signs
$\varepsilon_1,\ldots,\varepsilon_n\in \{-1, 1\}$ for which
every non-empty product (possibly with~repetitions) of
elements of~the~form $g_i^{\varepsilon_i}$ is non-trivial.
\end{lem}

\begin{proof} Choose $(\varepsilon_i)_{i=1}^n$ for which
$g_i^{\varepsilon_i}>e$ for $i=1,\ldots,n$. \end{proof}

Lemma \ref{ord_lem} is in fact equivalent to $G$ being
left-orderable,
but we will only need the implication we proved (cf.
\cite[Theorem 7.1.1]{rhemtulla}).

\section{Random groups with large number of generators}

We begin by reproducing terminology and useful observations
of \cite[Section~2]{przytycki11}.
By~$S$ we denote a finite set, called the \emph{alphabet}.
We define $S^{-1}$ to be the set of~formal inverses to~the~elements of~$S$, and
denote
$S^{\pm}=S\cup S^{-1}$. Elements of~$S^\pm$ are called the
\emph{letters}. By \emph{word over an alphabet $S$} we mean
a finite sequence of~\emph{letters}. We denote
$S=\{a_1,\ldots,a_n\}$, hence $n=|S|$.
$S$~is~to be interpreted as~the~set of generators of~$F_n$.

\begin{defin}[cf. {\cite[Definition 2.1]{przytycki11}}]
A \emph{basic automaton} (shortly
a~\emph{b-automaton}) \emph{over an alphabet $S$ with
transition data $\{\sigma_s\}$} is a pair
$(S,\{\sigma_s\})$,
where $\{\sigma_s\}_{s\in \{\emptyset\}\cup S^{\pm}}$ is a
family of subsets of~$S^\pm$.

The \emph{language} of a b-automaton with transition data
$\{\sigma_s\}$ is the set of~all non-empty words over $S$
beginning with a letter in $\sigma_\emptyset$
and such that for any two consecutive letters $ss'$ we have
$s'\in \sigma_s$.

We say that a b-automaton is \emph{$\lambda$-large}, for
some $\lambda \in (0,1)$, if
$\sigma_\emptyset \neq \emptyset$ and~for~each $s\in S^\pm$
we have $|\sigma_s|\geq \lambda 2n$.
\end{defin}

\begin{rem}[cf. {\cite[Remark
2.2(i)]{przytycki11}}]\label{bautomata_rmk1}
There are exactly $2^{2n(2n+1)}$ many b-automata over a
fixed alphabet $S$
of size $n$.
\end{rem}

\begin{rem}[cf. {\cite[Remark
2.2(ii)]{przytycki11}}]\label{bautomata_rmk2}
If a b-automaton is $\lambda$-large, then its~language
contains
at~least $\lceil \lambda 2n \rceil^{L-1}$ words of length
$L$ and at least $(\lceil \lambda2n \rceil -1)^{L-1}$
reduced words of length $L$.
\end{rem}

\begin{defin}[cf. {\cite[Definition 2.3]{przytycki11}}]

Let $I \subset \mathbb{N}_+$ be infinite and let
$\mathcal{L}$ be a set of reduced
words over an alphabet $S$, containing for all but finitely
many $L \in I$
at~least $ck^L$ words of length $L$, where $c>0, k>1$. Then
we~say that the \emph{$I$-growth rate of $\mathcal{L}$ is
at least $k$}. Similarly, if $k>k'$,
then we say that the \emph{$I$-growth rate of~$\mathcal{L}$
is greater than
$k'$}.

\end{defin}

It is convenient to extend the notion of density from
Definition \ref{random_group_def} in~the~following way.

\begin{defin}
Let $I \subset \mathbb{N}_+$ be infinite and let
$\mathcal{L}$ be a set
of reduced words over an alphabet $S$, containing for all
but finitely many $L \in I$
at least $c (2n-1)^{dL}$ words of~length $L$, where
$c>0, d\in (0,1)$.
Then we say that the~\emph{$I$-density of $\mathcal{L}$
is~at~least~$d$}.
\end{defin}

Notions of density $d$ and growth rate $k$ of the set
$\mathcal{L}$ are easily
seen to~be~strictly related by~$k=(2n-1)^d$, i.e. for such
$k,d$, with $d\in(0,1)$, the~set $\mathcal{L}$ has
$I$-growth rate at least $k$ if and only if it has
$I$-density at~least~$d$.

The following is a well known fact in random groups. We
reprove it in~a~stronger form in~Appendix
\ref{last_chapter}.

\begin{prop}[cf. {\cite[Section 9.A]{gromov93}}]
\label{dens_inters_prop}
Let $I\subset \mathbb{N}_+$ be infinite. Suppose
${d, d'\in (0,1)}$ are such that $d+d'>1$ and $R_f \subset
F_n$ is a fixed set of relators
in~some fixed number~$n$ of generators, of~$I$-density at
least $d'$.
Then \mbox{w.\ $I$-o.p.\ a} random set $R$ of~relators at
density $d$ intersects $R_f$.
\end{prop}

From this we get

\begin{lem}[cf. {\cite[Lemma
2.4]{przytycki11}}]\label{growth_dens_inters_lem}
Let $I\subset \mathbb{N}_+$ be infinite and let
$\mathcal{L}$ be a~set of~reduced words over the alphabet
$S$, \mbox{of  $I$-growth} rate greater than
$(2n-1)^{1-d}$, for some $d\in(0,1)$. Then w.\,$I$-o.p.\ a
random set
of relators at~density $d$ intersects~$\mathcal{L}$.
\end{lem}

We will be interested in the following consequence,
proven in~\cite{przytycki11}.

\begin{cor}[cf. {\cite[Corollary
2.5]{przytycki11}}]\label{cor_size}
For given $\lambda, d \in (0,1)$, if $n$ is sufficiently
large, then
w.o.p.\ a~random set of relators at density $d$ intersects
the languages of~all $\lambda$-large b-automata
over~the~alphabet $S$.
\end{cor}

For a group $G$ with presentation $G=\langle S|R \rangle$
and a word $w$ over the~alphabet $S$, we will denote by
$\overline{w}$ the corresponding element
of $G$.

To obtain Theorem \ref{target_thm} for $n$ sufficiently
large, we just need the following.

\begin{prop}\label{simple_non_ordered_prop}
Let $G$ be a group with presentation $G=\langle S|R
\rangle$ such that
$R$ intersects the~languages of all $\frac{1}{2}$-large
b-automata over $S$. Then $G$ has~no~non-trivial left-orderable quotients.
\end{prop}

In order to prove Proposition
\ref{simple_non_ordered_prop}, we use the following lemma,
which is~our main step towards exploiting the
hypothetical left-orderability.

\begin{lem}\label{positive_relator_lem}
Let $R$ be a set of words over $S$. Assume $R$
intersects the~languages of all~$\frac{1}{2}$-large
b-automata over $S$. Then for every choice of signs
$\varepsilon=(\varepsilon_1,\ldots,\varepsilon_n) \in
\{-1,1\}^n$ and~a~number $i\in\{1,\ldots,n\}$, there exists
a non-empty reduced word $w\in R$, consisting only
of~letters from the set
$\{a_1^{\varepsilon_1},a_2^{\varepsilon_2},\ldots,a_n^{\varepsilon_n}\}$,
with~at~least one
occurrence of $a_i^{\varepsilon_i}$. \end{lem}

\begin{proof}[Proof of Lemma \ref{positive_relator_lem}]
Consider a b-automaton $\mathbb{A}_{\varepsilon,i}$ over
$S$ with transition data
${\sigma_\emptyset=\{a_i^{\varepsilon_i}\}}$
and $\sigma_s =
\{a_1^{\varepsilon_1},\ldots,a_n^{\varepsilon_n}\}$
for every $s\in S^{\pm}$. Every word in its language
$\mathcal{L}_{\varepsilon, i}$ is reduced.
$\mathbb{A}_{\varepsilon, i}$ is $\frac{1}{2}$-large, hence
there exists some
$w \in \mathcal{L}_{\varepsilon, i} \cap R$. The word $w$
starts with $a_i^{\varepsilon_i}$ and~satisfies the
conditions we imposed.
\end{proof}

\begin{proof}[Proof of Proposition
\ref{simple_non_ordered_prop}]
Suppose there exists a~set of~relators $R'$,
containing $R$ and not necessarily finite,
such that $Q=\langle S| R' \rangle$
is left-orderable and~non-trivial. Let $a_{i_1},a_{i_2},\ldots,a_{i_m}$ be~all those $a_j\in S$, such that $\overline{a_j}\in Q$ is~non-trivial.
There must be at~least one, since $Q$ is generated by
the~elements of~the~form~$\overline{a_j}$.
By~Lemma~\ref{ord_lem}, we can find signs
${\varepsilon_{i_1},\ldots,\varepsilon_{i_m}\in\{-1,1\}}$,
such that every non-empty word consisting of~letters from
$\{a_{i_1}^{\varepsilon_{i_1}},\ldots,a_{i_m}^{\varepsilon_{
i_m}}\}$ represents a~non-trivial element of~$Q$. Note that those
words are always reduced.

Now for $j\in\{1,\ldots,n\}\setminus\{i_1,\ldots,i_m\}$
choose $\varepsilon_j\in\{-1,1\}$ in arbitrary way. We~have
thus defined a sequence $(\varepsilon_1,\ldots,\varepsilon_n) \in \{-1,1\}^n$.
By Lemma~\ref{positive_relator_lem} applied to this
sequence and~$i=i_1$, we~obtain
a word $w$ which lies in $R$, so~it~represents
the~trivial element of~$Q$, and~consists of~letters
of~the~form~$a_j^{\varepsilon_j}$ with~at~least one occurrence
of~$a_{i_1}^{\varepsilon_{i_1}}$. As $a_j^{\varepsilon_j}$
for $j \notin \{i_1,\ldots,i_m\}$ represent the~trivial
element, we can remove all occurrences of such letters from
$w$ and~obtain that way a word $w_1$, still representing
the~trivial element and~consisting only of letters
of~the~form~$a_{i_j}^{\varepsilon_{i_j}}$. $w_1$~is,~however,
non-empty because of~at least one occurrence of
$a_{i_1}^{\varepsilon_{i_1}}$. We arrive thus
at~a~contradiction with~the~earlier definition of signs
$\varepsilon_{i_1},\ldots,\varepsilon_{i_m}$.
\end{proof}

For fixed $d\in (0, 1)$ and $\lambda=\frac{1}{2}$ there is
$n_0$ such that the conclusion of~Corollary~\ref{cor_size}
holds for all $n\geq n_0$. For such $n$ Theorem
\ref{target_thm} is now almost immediate.

\begin{proof}[Proof of Theorem \ref{target_thm} for $n\geq
n_0$]
A random group $G$ at density $d$ is w.o.p.\ presented
by~$\langle S| R \rangle$,
where $R$ intersects the languages of all
$\frac{1}{2}$-large b-automata over~$S$, so, by
Proposition~\ref{simple_non_ordered_prop}, it has
no non-trivial left-orderable quotients.
\end{proof}

\section{Increasing the number of generators}

We now generalise our partial proof of Theorem
\ref{target_thm} to arbitrary
values of~$n\geq 2$. We follow closely the ideas of
{\cite[Section 3]{przytycki11}}.

We fix $n\geq 2$ and $d\in (0,1)$. We furthermore fix $B$ to
be a natural number that is~sufficiently large with respect
to $n$ and $d$
in a way we will specify later.

As before, we denote by $S$ the set of generators
$\{a_1,\ldots,a_n\}$.
Let $\widetilde{S}\subset F_n$ denote the~set of~reduced
words of length $B$ over the
alphabet $S$. The~involution on $\widetilde{S}$ mapping
each word to~its
inverse does not have fixed points. Thus we can partition
$\widetilde{S}$ into $\hat{S}$ and $\hat{S}^{-1}$.
We~introduce
the notation $\hat{S}^{\pm}=\hat{S}\cup\hat{S}^{-1}$ in
place of~$\widetilde{S}$.
Let $\hat{n}$ be the number $|\hat{S}|=n(2n-1)^{B-1}$.

Furthermore, for $0\leq P < B$ let $I_P \subset
\mathbb{N}_+$ denote the set of those $L$ that can be
written as~$L=B\hat{L}+P$ with $\hat{L}>0$.

\begin{defin}[{\cite[Definition 3.1]{przytycki11}}]
Let $r$ be a word of length $L\in I_0$ over
the~alphabet $S$. Divide
the word $r$ into $\hat{L}$ blocks of length $B$. This
determines a~new word $\hat{r}$ of length $\hat{L}$
over~the~alphabet $\hat{S}$, which we call
the~word \emph{associated} to $r$.
\end{defin}

\begin{defin}[{\cite[Definition
3.2]{przytycki11}}]\label{assoc}
Given a set $R$ of~reduced relators over $S$
of~equal length $L\in I_P$, we define
the \emph{associated group} $\hat{G}$ in the following way.

If $P=0$, then we consider the set $\hat{R}$ of relators
associated to relators in~$R$. We~define $\hat{G}$ to be
the group $\langle\hat{S}|\hat{R}\rangle$.

If $1\leq P < B$, then we do the following construction.
Suppose that ${r_1,r_2\in R}$ are two relators of length
$L$ over $S$, satisfying $r_1=q_1v^{-1}$ and
$r_2=vq_2$ (we assume $q_1,q_2,v$ to be reduced and that
there are no reductions
between $q_1$ and~$v^{-1}$ or~between $v$ and $q_2$), for
some word $v$ over $S$ of length $P$. We then obtain
a~(possibly non-reduced) word $q_1q_2$ over $S$, of~length
$2B\hat{L}$, with the property that $\overline{q_1q_2}=e$
in~$G=\langle S|R \rangle$. To this word we can associate,
as before, a relator over $\hat{S}$, of length $2\hat{L}$
(possibly
non-reduced), which we denote by~$\hat{r}(r_1,r_2)$. We
denote by $\hat{R}$ the set of~all $\hat{r}(r_1,r_2)$
as~above and we define $\hat{G}=\langle \hat{S} | \hat{R}
\rangle$.
\end{defin}

The main intuition here is that $\hat{R}$ obtained from a
random set $R$ of~relators
over $S$, at~density~$d$, at length $L\in I_0$ is very
similar
to a random set of relators over $\hat{S}$, at~the~same
density $d$,
at length $\frac{L}{B}$ (see
{\cite[Section~3]{przytycki11}}).

By increasing $B$, the number $\hat{n}$
can be made arbitrarily large. We can thus have $\hat{n}$
large
enough to obtain the conclusion of Corollary \ref{cor_size}
for intersections
of~languages of~$\frac{1}{2}$-large b-automata over
$\hat{S}$ with random sets of~relators at density $d$. We
then use the~following
analogue of Proposition~\ref{simple_non_ordered_prop}.

\begin{prop} \label{assoc_obs}
Suppose that $\hat{R}$, obtained as in Definition
$\ref{assoc}$ from $R$
being a set of reduced relators of the same length,
intersects languages of~all $\frac{1}{2}$-large b-automata
over $\hat{S}$. Then \mbox{$G=\langle S | R \rangle$}
has no non-trivial left-orderable quotients.
\end{prop}

\begin{proof}
Suppose there exists a~set of~relators $R'$,
containing $R$,
such that $Q=\langle S| R' \rangle$
is left-orderable and non-trivial. The construction
of $\hat{G}$ was performed in such a way, that by expanding
elements of $\hat{S}$ into words over $S$ we get a~natural
epimorphism $\phi: \hat{G} \twoheadrightarrow H$, where $H$ is the~subgroup of $Q$ generated by~the~elements
corresponding to~the~reduced words of length $B$ over~$S$.

We note that $H\subset Q$ is of finite index, since every
element $g\in Q$
is~of~the~form $g=\overline{w}$ for~some reduced word $w$ over
$S$
and~we may write $w=uv$ with $u$ of length at most $B$ and
$v$~of length divisible by $B$. We have thus
$\overline{v}\in H$,
hence $g \in \overline{u}H$ and~the~index $[Q:H]$ is not
greater
than the~number of possible values of $\overline{u}$, which
is finite.

Moreover, $H$ is non-trivial,
because otherwise $Q$ would be finite and non-trivial, hence
not~torsion-free, contradicting left-orderability (by
Corollary~$\ref{ord_cor}$).

Denote elements of
$\hat{G}$, represented by single letters from
$\hat{S}$, by $b_1,\ldots,b_{\hat{n}}$. They generate
$\hat{G}$,
so $H$ is generated by $\phi(b_1),\ldots,\phi(b_{\hat{n}})$,
not all of them being trivial. Let
$\phi(b_{i_1}),\ldots,\phi(b_{i_m})$
be all non-trivial elements of~the~form~$\phi(b_j)$.
The subgroup $H$ is~left-orderable, so, by Lemma
\ref{ord_lem}, there exist signs $\varepsilon_{i_1},\ldots,
\varepsilon_{i_m}\in \{-1,1\}$,
such that every non-empty product of~elements of~the~form
$\phi(b_{i_j})^{\eps_{i_j}}$ is non-trivial.
In~arbitrary way we choose $\eps_i\in\{-1,1\}$ for $i\in
\{1,\ldots,\hat{n}\}
\setminus\{i_1,\ldots,i_m\}$.

Fix $i=i_1$. For this index $i$ and the set $\hat{R}$ of
words over $\hat{S}$ we apply Lemma~\ref{positive_relator_lem} to conclude that there exists a
product of elements of~the~form~$b_j^{\eps_j}$, with
at~least one occurrence of~$b_{i_1}^{\eps_{i_1}}$, which
evaluates
to the trivial element in~$\hat{G}=\langle \hat{S} |
\hat{R} \rangle$.

By evaluating $\phi$ on this product, we get a product of
elements of
form $\phi(b_j)^{\eps_j}$, with at~least one occurrence of
$\phi(b_{i_1})^{\eps
_{i_1}}$, which evaluates to the~trivial element in~$H$.
Finally, by~leaving the~non-trivial factors only, we get
a~non-empty product of~elements of~the~form
$\phi(b_{i_j})^{\eps_{i_j}}$, evaluating to~the~trivial
element, which is a contradiction with
the definition of signs $\varepsilon_{i_1},\ldots,
\varepsilon_{i_m}$.

\end{proof}

The last element of the proof of Theorem \ref{target_thm}
is the following.
\begin{lem}[{\cite[Section
3]{przytycki11}}]\label{random_assoc}
If $B$ is sufficiently large, then,
in Gromov density model with $n$ generators, a~random set
$R$ of relators at density $d$ has w.o.p\ the~property,
that the set $\hat{R}$, obtained from~$R$ as in Definition
$\ref{assoc}$,
intersects languages of~all $\frac{1}{2}$-large b-automata
over $\hat{S}$.
\end{lem}

Assuming Lemma \ref{random_assoc}, the proof of Theorem
\ref{target_thm} is straightforward.

\begin{proof}[Proof of Theorem \ref{target_thm}]
Let $B$ be sufficiently large for the conclusion of 
Lemma \ref{random_assoc} to hold.
Then, by the combination of Proposition \ref{assoc_obs} and
Lemma \ref{random_assoc},
a~random group $G=\langle S |R \rangle$ in~Gromov density
model has w.o.p.\ no non-trivial left-orderable quotients.
\end{proof}

The proof of Lemma \ref{random_assoc} (in slightly stronger
form) is given in
{\cite[Section~3]{przytycki11}} in~the~first 5~lines of the
proof of {\cite[Theorem 1.5]{przytycki11}}. The hypothesis
of~{\cite[Proposition 2.6]{przytycki11}} for group
$\hat{G}=\langle \hat{S} | \hat{R} \rangle$,
obtained from a random group $G = \langle S | R \rangle$,
is checked there, which amounts to proving that $\hat{R}$,
obtained from a random set $R$ of relators in Gromov model,
intersects languages of all $\frac{1}{3}$-large b-automata
over $\hat{S}$. It remains
to note that all $\frac{1}{2}$-large b-automata are, in
particular, $\frac{1}{3}$-large.

\appendix

\section{Intersections of high-density
sets}\label{last_chapter}

From now on, by $I \subset\mathbb{N}_+$ we denote a fixed
infinite subset and all limits with $L\rightarrow \infty$
are taken over $L \in I$. The main result of
this appendix is the~following.

\begin{prop}\label{gen}
Suppose that for each $L\in I$ we have a set $R_L$ of size
$c_L>0$ with $a_L>0$ elements
distinguished. For fixed $L$ we pick uniformly and
independently at random entries of~a~$b_L$-tuple ($b_L>0$)
from $R_L$ and~obtain
this way a~random variable $D_L$ equal to the number
of~the~entries of~the~resulting tuple being distinguished.
Assume that
$\frac{a_Lb_L}{c_L} \rightarrow \infty$ as~$L \rightarrow
\infty$. Then for every $\varepsilon >0$ the~following
holds
\begin{equation*}
\lim\limits_{L
\rightarrow\infty}\mathbb{P}\left((1-\varepsilon)\frac{a_Lb_
L}{c_L} \leq
D_L \leq (1+\varepsilon)\frac{a_Lb_L}{c_L}\right)=1.
\end{equation*}
\end{prop}

Before proving Proposition \ref{gen}, let us use it to give
a proof of Proposition~\ref{dens_inters_prop} and
its~generalisation.

\begin{proof}[Proof of Proposition
\ref{dens_inters_prop}]
Let $c_L$, for $L\in I$, denote the number of all reduced
relators of~length $L$
over $S$, i.e. $c_L=|R_L|=2n(2n-1)^{L-1}$. Moreover, let
$a_L = |R_f \cap R_L|$ be~the~number of relators of length
$L$ we distinguish by~wanting them
to be selected in the random tuple. Let $b_L = \lfloor
(2n-1)^{dL} \rfloor$.
We assume $a_L \geq C (2n-1)^{d'L} $ for $L\in I$
sufficiently large and some $C>0$. At length $L$, $R$~is a
tuple
of $b_L$ elements, chosen uniformly and independently
at~random from $R_L$.
Let $D_L$ be as in Proposition $\ref{gen}$.
Note that $\frac{a_Lb_L}{c_L} \rightarrow \infty$ as $L
\rightarrow \infty$,
since $d+d' >1$. We~may thus apply Proposition \ref{gen}
for any $\varepsilon >0$
to see that a random set $R$ of relators at density $d$, at
length $L$ has w.\ $I$-o.p.\ at least $D_L \geq
(1-\varepsilon)\frac{a_Lb_L}{c_L}\geq K (2n-1)^{(d+d'-1)L}$
entries from $R_f$, for some $K>0$. For~$L$~sufficiently
large it clearly implies that $R$ and~$R_f$ intersect.
\end{proof}

If we moreover assume that $d<\frac{1}{2}$ and $R_f$ is
roughly (not just at least) of~density $d'$,
then we can prove that the intersection is roughly of
density $d+d'-1$.

\begin{prop}\label{birth_harder}
Suppose $d,d'\in(0,1)$ are such that $d+d'>1$ and
$d<\frac{1}{2}$. Let $R_f \subset F_n$
be~a~fixed set of relators in some fixed numer $n$ of
generators, such that
for some $C_1, C_2 >0$ the inequalities
\begin{equation}\nonumber
C_1(2n-1)^{d'L} \leq |R_f \cap R_L| \leq C_2(2n-1)^{d'L}
\end{equation}
hold for all sufficiently large $L\in I$.

Then for some $K_1,K_2>0$ a random set $R$ of relators at
density $d$, at~length~$L$ satisfies w.~$I$-o.p.\ the
inequalities
\begin{equation}\nonumber
K_1(2n-1)^{(d+d'-1)L} \leq |R_f \cap R| \leq
K_2(2n-1)^{(d+d'-1)L},
\end{equation}
where $|R_f \cap R|$ denotes the number of distinct entries
of $R$, belonging
to $R_f$.
\end{prop}

\begin{proof}
We use the notation from the proof of Proposition
\ref{dens_inters_prop}. Analogously to that proof, for some
$K_1,K_2>0$ we obtain
\begin{equation}\label{KK}
K_1(2n-1)^{(d+d'-1)L} \leq D_L \leq K_2(2n-1)^{(d+d'-1)L},
\end{equation} occurring w.\ $I$-o.p.\

Since $d<\frac{1}{2}$, we have
$\frac{b_L^2}{c_L}\rightarrow 0$ as $L\rightarrow \infty$.

Let us estimate the probability $q_L$ that in the
experiment defining $D_L$ all~elements of~the~obtained
$b_L$-tuple are pairwise distinct. It is the~same
as~the~probability that every element
of~the~tuple is different from the elements having smaller
indices (we assume
some fixed order on a tuple), so

\begin{equation}\nonumber
q_L=1\left(1-\frac{1}{c_L}\right)\left(1-\frac{2}{c_L}\right
)\ldots\left(1-\frac{b_L-1}{c_L}\right)
\geq \left(1-\frac{b_L-1}{c_L}\right)^{b_L}.
\end{equation}

For $L\in I$ sufficiently large we have $\frac{b_L^2}{c_L}
<1$, so $b_L\leq b_L^2 < c_L$
and the~number $x_L = -\frac{b_L-1}{c_L}$ sastifies $x_L
\geq -1$.
It means that we can use Bernoulli's inequality
to~obtain

\begin{equation}\nonumber
q_L \geq
\left(1+\left(-\frac{b_L-1}{c_L}\right)\right)^{b_L} \geq
1-b_L\frac{b_L-1}{c_L}.
\end{equation}

Obviously, $b_L\frac{b_L-1}{c_L} \rightarrow 0$ as
$L\rightarrow \infty$, because
$\frac{b_L^2}{c_L} \rightarrow 0$ as $L\rightarrow \infty$.
It~follows that
$q_L \rightarrow 1$ as~$L\rightarrow \infty$, so w.\
$I$-o.p.\ the number $D_L$ is the
number of~distinct entries of~$R$ belonging to $R_f$, which
combined with $(\ref{KK})$
concludes the~proof.
\end{proof}

For the proof of Proposition \ref{gen} we will apply the
following bound known as~the~Chebyshev's inequality.

\begin{lem}[{\cite[Lemma 3.1]{kallenberg}}]\label{cheb}
If $\xi$ is a random variable with $\mathbb{E}\xi^2<\infty$,
then for~every $\alpha >0$

\begin{equation}\nonumber
\mathbb{P}\left( |\xi - \mathbb{E}\xi| \geq \alpha \right)
\leq \frac{\Var \xi}{\alpha^2}.
\end{equation}
\end{lem}

\begin{proof}[Proof of Proposition \ref{gen}]

Fix $L\in I$. For $i=1,\ldots,b_L$ denote by $X_i^{(L)}$
the~random variable equal to 1, if $i$-th
element of~the considered random tuple is distinguished,
and equal to 0, otherwise.\\
Variables $(X_i^{(L)})_{i}$ are independent and
$\P(X_i^{(L)}=1)=\frac{a_L}{c_L}=1-\P(X_i^{(L)}=0)$,
so~$\E X_i^{(L)}=\frac{a_L}{c_L}$.
Next we check that $\Var
X_i^{(L)}=\frac{a_L}{c_L}(1-\frac{a_L}{c_L})$.
Since $D_L = \sum\limits_{i=1}^{b_L}X_i^{(L)}$, we have $\E
D_L = \frac{a_Lb_L}{c_L}$
and~$\Var D_L = \frac{a_Lb_L}{c_L}(1-\frac{a_L}{c_L})$.
Fix $\varepsilon >0$. Now we apply Lemma~\ref{cheb} for
$\xi = D_L$ and $\alpha = \varepsilon \mathbb{E}D_L$,
obtaining

\begin{equation}\nonumber
\begin{split}
\mathbb{P}\left( \biggl| D_L-\frac{a_Lb_L}{c_L} \biggr|
\geq \varepsilon \frac{a_Lb_L}{c_L}
\right)
&=
\mathbb{P}\left( \bigr| D_L-\mathbb{E}D_L \bigr|
\geq \varepsilon\mathbb{E}D_L
\right)\\
&\leq \frac{\Var D_L}{(\varepsilon\mathbb{E}D_L)^2}
=\frac{1-\frac{a_L}{c_L}}{\varepsilon^2\frac{a_Lb_L}{c_L}}
\leq \frac{1}{\varepsilon^2 \frac{a_Lb_L}{c_L}}
\rightarrow 0
\end{split}
\end{equation}

as $L\rightarrow \infty$, since we assumed that
$\frac{a_Lb_L}{c_L} \rightarrow \infty$.

\end{proof}

\subsection*{Acknowledgements}

Yago Antolin-Pichel, in private communication, suggested the
use of~Lemma \ref{ord_lem}, which is the key
property of the left-ordered groups used in the proof.
Piotr Przytycki suggested the~topic of this paper,
encouraged and advised the~author.
I~would like to thank both of them.


\begin{thebibliography}{10}

%% Use the widest label as parameter.
%% Reference items can be numbered or have labels of your choice, as below.

%% In IMPAN journals, only the title is italicized; boldface is not used.
%% Our software will add links to many articles; for this, enclosing volume numbers in { } is helpful
%% Do not give the issue number unless the issues are paginated separately.

%%%%%%%%%%% To ease editing, use normal size:

\normalsize
\baselineskip=17pt

%%%%%%%%%%%%%

\bibitem{rhemtulla}
R. Botto Mura, A. Rhemtulla,
\emph{Orderable groups}.
Lecture Notes in Pure and Applied Mathematics,
Vol. {27},
Marcel Dekker, New York-Basel,
1977.

\bibitem{przytycki11}
F. Dahmani, V. Guirardel, P. Przytycki,
\emph{Random groups do not split}.
Mathematische Annalen {349}(3) (2011),
657-673.

\bibitem{deroin14}
B. Deroin, A. Navas, C. Rivas,
\emph{Groups, orders, and dynamics}.
arXiv:1408.5805, 2014.

\bibitem{furman11}
A. Furman,
\emph{A survey of measured group theory}.
in:~Geometry, Rigidity, and~Group Actions.
Chicago Lectures in Mathematics Series,
The~University of~Chicago Press,
Chicago, 2011, 296-374.

\bibitem{gromov93}
M. Gromov,
\emph{Asymptotic invariants of infinite groups}.
in:~Geometric Group Theory, vol. {2} (Sussex, 1991).
London Mathematical Society Lecture Note Series, vol. {182},
Cambridge University Press, Cambridge, 1993, 1-295.

\bibitem{kallenberg}
O. Kallenberg,
\emph{Foundations of Modern Probability},
Springer Series in~Statistics. Probability
and its applications, Springer-Verlag New~York,
New~York, 1997.

\bibitem{ollivier04}
Y. Ollivier,
\emph{Sharp phase transition theorems for hyperbolicity of random groups}.
Geom. Funct. Anal. {14}(3) (2004), 595-679.

\bibitem{ollivier05}
Y. Ollivier,
\emph{A January 2005 invitation to random groups}.
in:~Ensaios Matem\'{a}ticos (Mathematical Surveys),
vol. {10}. Sociedade Brasileira de~Matem\'{a}tica,
Rio de Janeiro, 2005.

\end{thebibliography}
\end{document}